\documentclass[10pt]{amsproc}
\usepackage[dvips]{epsfig}
\usepackage{amsmath}
\usepackage{amsfonts}
\usepackage[psamsfonts]{amssymb}
\usepackage{amsthm}
\usepackage[mathscr]{eucal}
\usepackage{exscale}
\usepackage{graphics,graphicx,color}
\usepackage[T1,T2A]{fontenc}
\usepackage[cp1251]{inputenc}
\usepackage[russian,english]{babel}

\newcommand{\ccomma}{\mathpunct{\raisebox{0.5ex}{,}}}

\newtheorem{theorem}{Theorem}
\newtheorem{lemma}{Lemma}[section]
\newtheorem{definition}{Definition}
\newtheorem{remark}{Remark}
\newtheorem{corollary}{Corollary}

\def\theequation{\arabic{section}.\arabic{equation}}

\begin{document}
\title[On summation of the Taylor series by the theta summation method]
{On summation of the Taylor series of the\\[0.5ex]
 function {\large 1}{\large /}{\large (1}-{\Large\textit{z}}{\large )}
by the theta summation method.}
\author{Victor Katsnelson}
\address{Department of Mathematics, the Weizmann Institute, Rehovot 76100, Israel}
\email{victor.katsnelson@weizmann.ac.il, victorkatsnelson@gmail.com}
\thanks{}

\date{February, 19, 2014}

\dedicatory{}


\begin{abstract}
\label{LPS}
The family of the Taylor series
\begin{math}
f_{\varepsilon}(z)=
\sum\limits_{0\leq{}n<\infty}e^{-\varepsilon{}n^2}z^n
\end{math}
is considered, where the parameter \(\varepsilon\), which enumerates the family, runs over
\(]0,\infty[\). For each fixed \(\varepsilon>0\), this Taylor series converges locally uniformly
with respect to \(z\in\mathbb{C}\) and represents an entire function in \(z\) of zero order.
The limiting behavior of the family \(\{f_{\varepsilon}(z)\}_{0<\varepsilon<\infty}\) is studied as \(\varepsilon\to+0\). Let \(\mathscr{G}\) be the interior of the closed curve
\(\mathscr{C}=\{\zeta\in\mathbb{C}:\,\zeta=e^{|t|+it},\,t\in[-\pi,\pi]\,\,\,\}\). It was shown that
\(\lim\limits_{\varepsilon\to+0}f_{\varepsilon}(z)=1/(1-z)\) for \(z\in\mathscr{G}\) locally uniformly
with respect to \(z\). Moreover,
\(\varlimsup\limits_{\varepsilon\to+0}|f_{\varepsilon}(z)|=\infty\) for \(z\not\in\mathscr{G}\).
\end{abstract}

\maketitle

\noindent
\textbf{Notation}:\\
\(\mathbb{C}\) stands for the complex plane;\\
\(\mathbb{R}\) stands for the real axis;\\
\(\mathbb{Z}\) stands for the set of all integers;\\
\(\mathbb{T}\) stands for the unit circle: \(\mathbb{T}=\{z\in\mathbb{C}:\,|z|=1\}\);\\
\(\mathbb{D}\) stands for the open unit disc: \(\mathbb{D}=\{z\in\mathbb{C}:\,|z|<1\}\); \\
\(\overline{\mathbb{D}}\) stands for the closed unit disc: \(\overline{\mathbb{D}}=\mathbb{D}\cup\mathbb{T}=\{z\in\mathbb{C};\,|z|\leq1\);\\
\(\mathbb{D}^{-}\) stands for the exterior of the unit circle \(\mathbb{T}\): \(\mathbb{D^{-}}\!\!=\!\{z\in\mathbb{C}:\,|z|>1\}\cup\{\infty\}\);\\
\(\{z\}\) stands for the one-point set which consists of the point \(z\).

\section{Summation methods of Taylor Series}
\label{Sec1}
\setcounter{equation}{0}
Let \(f(z)\) be a function holomorphic in the unit disc \(\mathbb{D}\).
Such a function \(f(z)\) can be expanded in the Taylor series
\begin{equation}
\label{TSE}
f(z)=\sum\limits_{n=0}^{\infty}a_nz^n.
\end{equation}
The series on the right hand side of the equality \eqref{TSE} converges uniformly on every
compact subset of the disk \(\mathbb{D}\).  Assume that the radius of convergence of this
series is equal to one, that is
\begin{equation}
\label{CHad}
\varlimsup\limits_{n\to\infty}\sqrt[n]{|a_n|}=1\,.
\end{equation}
Then the function \(f\) can not be extended as a holomorphic function to any disc
\(\lbrace{}z:\,|z|<R\rbrace\) with \(R>1\). In other words, on the unit circle
\(\mathbb{T}\) there is at least one singular point of the function \(f\).

Assume moreover that the function \(f\) is holomorphic in some domain~\(\mathscr{D}\),
\begin{equation*}
\mathbb{D}\subset\mathscr{D}\subset\mathbb{C}\,.
\end{equation*}
Of course, the boundary \(\partial\mathscr{D}\) of  \(\mathscr{D}\) must intersect
with \(\mathbb{T}:\,\partial\mathscr{D}\cap\mathbb{T}\not=\emptyset\). Otherwise
the radius of convergence of the Taylor series \eqref{TSE} will be greater than one.

The question arises. Can the Taylor series \eqref{TSE} be summed to the function
\(f\) in some domain \(\mathscr{G}\) larger than the unit disc \(\mathbb{D}\):
\begin{equation}
\label{geom}
\mathbb{D}\subset\mathscr{G}\subseteq\mathscr{D}\,.
\end{equation}

A summation method is determined by a sequence
\(\lbrace\gamma_n(\varepsilon)\rbrace_{0\leq{}n<\infty}\), where for each \(n=0,\,1,\,2,\,3\,.\ldots\), \(\gamma_n(\varepsilon)\) is a complex valued function
defined for \(0<\varepsilon<\varepsilon_0\), \(0<\varepsilon_0\leq+\infty\),
and the following conditions are satisfied:\\[1.0ex]
\begin{enumerate}
\item[\textsf{a.}]\(\hfill\sup\limits_{\substack{0<\varepsilon<\varepsilon_0\\0\leq{}n<\infty}} |\gamma_n(\varepsilon)| <\infty\,,   \hfill\)\\[1.0ex]
\item[\textsf{b.}]\(\hfill\lim\limits_{\varepsilon\to+0}\gamma_n(\varepsilon)=1\ \ \textup{for each}\ \ n=0,\,1,\,2,\,3,\,\ldots\,, \hfill\)\\[1.0ex]
\item[\textsf{c.}]\(\hfill\lim\limits_{n\to\infty}\sqrt[n]{|\gamma_n(\varepsilon)|}=0
\ \ \textup{for each} \ \ \varepsilon\in(0,\epsilon_0)\,.\hfill\)\\[1.0ex]
\end{enumerate}

Such a sequence \(\lbrace\gamma_n(\varepsilon)\rbrace_{0\leq{}n<\infty}\) is said to be
\emph{a summing sequence}.

Let \(f_{\varepsilon}(z)\) be the function defined by the power series
\begin{equation}
\label{feps}
f_{\varepsilon}(z)=\sum\limits_{0\leq{}n<\infty}\gamma_n(\varepsilon)a_nz^n\,.
\end{equation}
The conditions \eqref{CHad} and \textsf{c.} ensure that the radius of convergence of the
power series in \eqref{feps} is equal to infinity. Thus the function \(f_{\varepsilon}(z)\)
is an entire function for each \(\varepsilon>0\).
The equality \eqref{TSE} and the conditions \eqref{CHad}, \textsf{a.} and \textsf{b.}
ensure that
\begin{equation}
\label{ctf}
\lim_{\varepsilon\to+0}f_{\varepsilon}(z)=f(z)
\end{equation}
locally uniformly for \(z\in\mathbb{D}\).

If the limiting relation \eqref{ctf} holds locally uniformly in the domain \(\mathscr{G}\),
\eqref{geom}, than we say that the Taylor series of the function \(f\) is summable to the
function \(f\) by the summation method \(\lbrace\gamma_n(\varepsilon)\rbrace_{0\leq{}n<\infty}\) in the domain \(\mathscr{G}\).

The following three summing  sequences are well known (see \cite[Notes on Chapter VIII,\,\S8.10]{Har}):
\begin{subequations}
\label{CSS}
\begin{gather}
\label{LRSS}
\gamma_n(\varepsilon)=\frac{\Gamma(1+\varepsilon{}n)}{\Gamma(1+n)},\ \ n=0,\,1,\,2,\,\,\ldots\,,\\
\label{LinSS}
\gamma_0(\varepsilon)=1,\ \ \gamma_n(\varepsilon)=e^{-\varepsilon{}n\log n}, \ \ n=1,\,2,\,\,\ldots\,,\\
\gamma_n(\varepsilon)=\frac{1}{\Gamma(1+\varepsilon n)}, \ \ n=0,\,1,\,2,\,\,\ldots\,.
\label{MLSS}
\end{gather}
\end{subequations}
Each of these three sequences sums the Taylor series \(\sum\limits_{0\leq n<\infty}z^n\) of the function \(f(z)=\frac{1}{1-z}\) to this function in the domain \(\mathbb{C}\setminus[1,+\infty[\):
\begin{equation}
\label{SSSF}
\frac{1}{1-z}=\lim_{\varepsilon\to+0}\sum\limits_{0\leq n<\infty}\gamma_n(\varepsilon)z^n, \ \ \textup{locally uniformly for} \ \ z\in \mathbb{C}\setminus[1,+\infty[\,.
\end{equation}
The result \eqref{SSSF} can be applied to the summing procedure \eqref{feps}-\eqref{ctf}
applied to an \emph{arbitrary} function \(f(z)\) whose Taylor series \eqref{TSE} has a positive radius of convergence. The domain \(\mathscr{G}\), where the limiting relation \eqref{ctf} holds, is the so called \emph{Mittag-Leffler star} of the function \(f\).
See \cite[Chapter VIII,\,\S\,8.10, Theorem 135.]{Har}. (The Mittag-Leffler star of the function \(\frac{1}{1-z}\) is the domain \(\mathbb{C}\setminus[1,+\infty[\).)

\begin{remark}
The series on the right hand side of \eqref{SSSF} is of the form
\begin{equation}
\label{SSSf}
\sum\limits_{0\leq n<\infty}A(n)z^n,
\end{equation}
where \(A(\zeta)\) is an entire function. The series of the form \eqref{SSSf} is a very classical
subject. They were considered since the 90s of the 19th century. In particular assuming that
the radius of convergence of the series \eqref{SSSf} is finite, the question on the analytic continuation of the series from the disc of convergence to a larger domain was studied.
See for example the classical papers \cite{Lea} and \cite{LeR}. The book \cite{Bie} of
L.\,Bieberbach is a fount of wisdom on the analytic continuation of Taylor series.
\end{remark}
\section{Theta summation method. Convergence.}
\label{Sec2}
\setcounter{equation}{0}
In this paper we discuss only one special summing sequence:
\begin{equation}
\label{SpSS}
\gamma_n(\varepsilon)=e^{-\varepsilon{}n^2}, \ \ n=0,\,1,\,2,\,\ldots\,.
\end{equation}
 he summation method corresponding to the summing
sequence \eqref{SpSS} is said to be the theta summation method.

We apply the theta summation method the Taylor series  \(\sum\limits_{0\leq n<\infty}z^n\) of the function \(f(z)=\frac{1}{1-z}\).

We succeeded in finding a precise answer to the following question: \textit{for what \(z\in\mathbb{C}\setminus\{1\}\)
the limiting relation
\begin{subequations}
\label{TeSR}
\begin{align}
\label{TeSRa}
\frac{1}{1-z}=\lim_{\varepsilon\to+0}f_{\varepsilon}(z),\\
\intertext{\textit{holds, where}}
f_\varepsilon(z)\stackrel{\scriptscriptstyle{\text{def}}}{=}
\sum_{0\leq{}n<\infty}\!\!e^{-\varepsilon{}n^2}z^n\,.
\label{TeSRb}
\end{align}
\end{subequations}
}

\begin{lemma}
\label{L11} Given \(\varepsilon>0\), the series \eqref{TeSRb} converges for every \(z\in\mathbb{C}\)
locally uniformly with respect to \(z\). For each \(\varepsilon>0\), the function \(f_\varepsilon(z)\) defined by this series
is an entire function of zero order:
 \begin{equation}
 \label{ZeOr}
 \lim\limits_{r\to\infty}\frac{\ln\ln{}M_{f_\varepsilon}(r)}{\ln r}=0,\ \ \text{where} \ \
 M_{f_\varepsilon}(r)=\max\limits_{z:|z|\leq{}r}|f_\varepsilon(z)|\,.
 \end{equation}
 \end{lemma}

In contrast to the cases of the limiting relations \eqref{SSSF} with \(\gamma_n(\varepsilon)\) of the form \eqref{CSS}, the set \(\mathscr{G}\) of those \(z\), where the limiting relation \eqref{TeSRa}
holds, is essentially smaller than the domain \(\mathbb{C}\setminus[1,+\infty[\). (Nevertheless, still \(\mathscr{G}\supset\mathbb{D},\,\mathscr{G}\not=\mathbb{D}\).)

The s  tarting point of our reasoning is the following Fourier representation of the
summing sequence
\(\lbrace{e^{-\varepsilon{}n^2}}\rbrace_{0\leq{}n<\infty}\):
\begin{equation}
\label{IRSS}
e^{-\varepsilon{}n^2}=\frac{1}{2\sqrt{\pi\varepsilon}}
\int\limits_{-\infty}^{+\infty}e^{-\xi^2/4\varepsilon}
e^{in\xi}\,d\xi\,, \ \ \varepsilon>0\,.
\end{equation}
Substituting \eqref{IRSS} into \eqref{TeSRb}, we obtain the following representation
for the function \(f_\varepsilon(z)\), \eqref{TeSRb},
\begin{equation}
\label{IRFe}
f_\varepsilon(z)=\frac{1}{2\sqrt{\pi\varepsilon}}\sum_{0\leq{}n<\infty}
\int\limits_{-\infty}^{+\infty}e^{-\xi^2/4\varepsilon}
e^{in\xi}z^n\,d\xi\,, \ \ \varepsilon>0\,,
\end{equation}
which holds for arbitrary \(z\in\mathbb{C}\). For \(z\in\mathbb{D}\), we can change order
of summation and integration in \eqref{IRFe}. Thus
\begin{equation}
\label{IRFee}
f_\varepsilon(z)=\frac{1}{2\sqrt{\pi\varepsilon}}
\int\limits_{-\infty}^{+\infty}e^{-\xi^2/4\varepsilon}\frac{1}{1-ze^{i\xi}}
\,d\xi\,, \ \ \varepsilon>0\,,\,z\in\mathbb{D}.
\end{equation}
Splitting the integral on the right hand side of \eqref{IRFee}, we obtain
\begin{equation}
\label{SpI}
  f_\varepsilon(z)=f^{+}_{\varepsilon}(z)+f^{-}_{\varepsilon}(z),
\end{equation}
where
\begin{subequations}
\label{SpIn}
\begin{align}
\label{SpIna}
f^{+}_{\varepsilon}(z)&=\frac{1}{2\sqrt{\pi\varepsilon}}
\int\limits_{0}^{+\infty}e^{-\xi^2/4\varepsilon}
\frac{1}{1-ze^{i\xi}}\,d\xi,&&\varepsilon>0\,,
\,z\in\mathbb{D}\,,\\
\label{SpInb}
f^{-}_{\varepsilon}(z)&=\frac{1}{2\sqrt{\pi\varepsilon}}
\int\limits_{0}^{+\infty}e^{-\xi^2/4\varepsilon}
\frac{1}{1-ze^{-i\xi}}\,d\xi,&&\varepsilon>0\,,
\,z\in\mathbb{D}\,.
\end{align}
\end{subequations}
Both the integrals in \eqref{SpIn} are taken over the ray \([0,+\infty[\). From the formulas \eqref{SpIn} is evident that each of the functions \(f^{+}_{\varepsilon}(z)\),
\(f^{-}_{\varepsilon}(z)\) is holomorphic in the unit disc \(\mathbb{D}\). Moreover,
\begin{subequations}
\label{BLR}
\begin{align}
\label{BLRa}
\lim_{\varepsilon\to+0}f^{+}_{\varepsilon}(z)&=\frac{1}{2}\cdot\frac{1}{1-z}\,, \\
\label{BLRb}
\lim_{\varepsilon\to+0}f^{-}_{\varepsilon}(z)&=\frac{1}{2}\cdot\frac{1}{1-z}\,,
\end{align}
\end{subequations}
for \(z\in\mathbb{D}\) locally uniformly in \(\mathbb{D}\).

It turns out that each of the functions of the family
\(\lbrace{}f_{\varepsilon}^{+}\rbrace_{\varepsilon>0}\)
can be continued analytically from the unit disc \(\mathbb{D}\) to a domain
\(\mathscr{G}^{+}\), \(\mathscr{G}^{+}\supset\mathbb{D}\)\,, and
each of the functions of the family
\(\lbrace{}f_{\varepsilon}^{-}\rbrace_{\varepsilon>0}\)
can be continued analytically from the unit disc \(\mathbb{D}\) to a domain
\(\mathscr{G}^{-}\), \(\mathscr{G}^{-}\supset\mathbb{D}\)\,.
We describe the domains \(\mathscr{G}^{+}\) and \(\mathscr{G}^{-}\) as follows.
Let \(\mathscr{S}^{+}\) and \(\mathscr{S}^{-}\) be spiral-shaped curves:
\begin{subequations}
\label{SpCu}
\begin{align}
\label{SpCu+}
\mathscr{S}^{+}=\lbrace\zeta\in\mathbb{C}:\,\zeta=&e^{(1-i)t},\,
0\leq{}t<\infty\rbrace\,,\\
\label{SpCu-}
\mathscr{S}^{-}=\lbrace\zeta\in\mathbb{C}:\,\zeta=&e^{(1+i)t},\,
0\leq{}t<\infty\rbrace\,.
\end{align}
\end{subequations}
By definition,
\begin{equation}
\label{DefDom}
\mathscr{G}^{+}=\mathbb{C}\setminus\mathscr{S}^{+}\,, \hspace{4.0ex}
\mathscr{G}^{-}=\mathbb{C}\setminus\mathscr{S}^{-}\,.
\end{equation}
It is clear that the domains \(\mathscr{G}^{+}\) and \(\mathscr{G}^{-}\) are simply
connected, and
\begin{equation}
\label{DSub}
\mathbb{D}\subset\mathscr{G}^{+},\hspace{4.0ex}\mathbb{D}\subset\mathscr{G}^{-}\,.
\end{equation}
To see that the functions \(f^{+}_{\varepsilon}(z)\),\,\(f^{-}_{\varepsilon}(z)\)
can be continued analytically from the disc \(\mathbb{D}\) to the domains
\(\mathscr{G}^{+}\) and \(\mathscr{G}^{-}\) respectively, we modify the integral representations  \eqref{SpIn} of these functions rotating a ray of integration.
For fixed \(z\in\mathbb{D}\) and \(\varepsilon>0\), the function
\(\displaystyle{e^{-\xi^2/4\varepsilon}\frac{1}{1-ze^{i\xi}}}\) which appears in \eqref{SpIna}
is holomorphic with respect to \(\xi\) within the angular domain
\(0\leq\arg\zeta\leq\frac{\pi}{4}\). Moreover for each fixed
\(\vartheta,\,0<\vartheta<\frac{\pi}{4}\), this function is fast decaying as \(|\xi|\to\infty\) within the angular domain \(0\leq\arg\xi\leq\vartheta\):
\begin{multline}
\label{FaDe}
\hfill\left|e^{-\xi^2/4\varepsilon}\frac{1}{1-ze^{i\xi}}\right|\leq
\frac{1}{1-|z|}e^{-|\xi|^2\cos2\vartheta/4\varepsilon}\,,\hfill\hfill\hfill\\
z\in\mathbb{D},\hspace{1.5ex} 0\leq|\xi|<\infty,0\leq\arg\xi\leq\vartheta\,.
\end{multline}
 Therefore in \eqref{SpIna} we can rotate a ray of integration counterclockwise :
 \begin{subequations}
 \label{SpRo}
 \begin{multline}
 \label{SpRo+}
 f^{+}_{\varepsilon}(z)=\frac{e^{i\vartheta}}{2\sqrt{\pi\varepsilon}}
\int\limits_{0}^{+\infty}
e^{-\xi^2\textstyle{e}^{\scriptscriptstyle{2i\vartheta}}/4\varepsilon}
\frac{1}{1-ze^{i\xi{}\textstyle{e}^{\scriptscriptstyle{i\vartheta}}}}\,d\xi{}\\
=\frac{e^{i\vartheta}}{2\sqrt{\pi}}
\int\limits_{0}^{+\infty}
e^{-\xi^2\textstyle{e}^{\scriptscriptstyle{2i\vartheta}}/4}
\frac{1}{1-ze^{i\sqrt{\varepsilon}\xi{}\textstyle{e}^{\scriptscriptstyle{i\vartheta}}}}\,d\xi, \hspace{1.5ex}\varepsilon>0\,,\hspace{0.5ex}
\,z\in\mathbb{D}\,,\hspace{0.5ex}0\leq\vartheta<\frac{\pi}{4}\,.
 \end{multline}
 Analogously  in \eqref{SpInb} we can rotate a ray of integration clockwise :
 \begin{multline}
 \label{SpRo-}
 f^{-}_{\varepsilon}(z)=\frac{e^{-i\vartheta}}{2\sqrt{\pi}}
\int\limits_{0}^{+\infty}
e^{-\xi^2\textstyle{e}^{\scriptscriptstyle{-2i\vartheta}}/4}
\frac{1}{1-ze^{i\sqrt{\varepsilon}\xi{}\textstyle{e}^{\scriptscriptstyle{-i\vartheta}}}}\,d\xi,\\ {}\varepsilon>0\,,
\,z\in\mathbb{D}\,,\hspace{1.0ex}0\leq-\vartheta<\frac{\pi}{4}\,.
 \end{multline}
 \end{subequations}
 For \(\vartheta:\,0<\vartheta<\frac{\pi}{2}\), let
 \(\mathscr{S}_{\vartheta}^{+}\) and \(\mathscr{S}_{\vartheta}^{-}\) be spiral curves
 \begin{subequations}
\label{SpCuT}
\begin{align}
\label{SpCuT+}
\mathscr{S}_{\vartheta}^{+}=\lbrace\zeta\in\mathbb{C}:\,\zeta=&e^{(\tg\vartheta-i)t},\,
0\leq{}t<\infty\rbrace\,,\\
\label{SpCuT-}
\mathscr{S}_{\vartheta}^{-}=\lbrace\zeta\in\mathbb{C}:\,\zeta=&e^{(\tg\vartheta+i)t},\,
0\leq{}t<\infty\rbrace\,,
\end{align}
\end{subequations}
and let \(\mathscr{G}_{\vartheta}^{+}\) and
\(\mathscr{G}_{\vartheta}^{-}\) be the sets
\begin{equation}
\label{DefDomT}
\mathscr{G}_{\vartheta}^{+}=\mathbb{C}\setminus\mathscr{S}_{\vartheta}^{+}\,, \hspace{4.0ex}
\mathscr{G}_{\vartheta}^{-}=\mathbb{C}\setminus\mathscr{S}_{\vartheta}^{-}\,.
\end{equation}
(In this notation, the sets \(\mathscr{S}^{\pm}, \mathscr{G}^{\pm}\) introduced in
\eqref{SpCu}-\eqref{DefDom} are \(\mathscr{S}^{\pm}_{\pi/4},\,\mathscr{G}^{\pm}_{\pi/4}\)).\\
For \(0<\theta<\frac{\pi}{2}\), each of the sets \(\mathscr{G}_{\vartheta}^{+}\),\,\(\mathscr{G}_{\vartheta}^{-}\) is a connected
open set (a domain) in \(\mathbb{C}\) containing the unit disc:
\begin{equation}
\mathbb{D}\subset\mathscr{G}_{\vartheta}^{+},\hspace{1.5ex}
\mathbb{D}\subset\mathscr{G}_{\vartheta}^{-},\hspace{1.5ex}0<\vartheta<\tfrac{\pi}{2}\,.
\end{equation}
The curve \(\mathscr{S}_{\vartheta}^{+}\) is the boundary of the domain \(\mathscr{G}_{\vartheta}^{+}\).
When \(\xi\) runs over the positive half-axis \([0,+\infty[\)\,, the point
\(e^{-i\varepsilon\xi{}\textstyle{e}^{\scriptscriptstyle{i\vartheta}}}\) runs over the curve
\(\mathscr{S}_{\vartheta}^{+}\) for every fixed \(\varepsilon>0\). Therefore for \(z\in\mathscr{G}_{\vartheta}\) and \(\xi\in[0,\infty[\), the value
\(\big|e^{-i\varepsilon\xi{}\textstyle{e}^{\scriptscriptstyle{i\vartheta}}}-z\big|\) is bounded away
from zero:
\begin{equation*}
\left|e^{-i\varepsilon\xi{}\textstyle{e}^{\scriptscriptstyle{i\vartheta}}}-z\right|
\geq\textup{dist}\,(z,\mathscr{S}_{\vartheta}^{+})>0,\hspace{1.5ex}
z\in\mathscr{G}_{\vartheta}^{+}\,,\ \ 0\leq\xi<\infty.
\end{equation*}
(Here \(\textup{dist}\,(z,\mathscr{S}_{\vartheta}^{+})\) is the distance from the point
\(z\in\mathscr{G}_{\vartheta}^{+}\) to the set \(\mathscr{S}_{\vartheta}^{+}\).)
Thus, for fixed \(\vartheta\in[\vartheta,\frac{\pi}{2}[\) \ and \(z\in\mathscr{G}_{\vartheta}^{+}\),
the function \(\displaystyle{\frac{1}{1-ze^{i\varepsilon\xi{}
\textstyle{e}^{\scriptscriptstyle{i\vartheta}}}}}=
1+\frac{z}{e^{-i\varepsilon\xi{}\textstyle{e}^{\scriptscriptstyle{i\vartheta}}}-z}\) of the variable \(\xi\) is bounded on the positive
half-axis \(0\leq\xi<\infty\)\,:
\begin{equation}
\label{fet+}
\left|\frac{1}{1-ze^{i\varepsilon\xi{}
\textstyle{e}^{\scriptscriptstyle{i\vartheta}}}}\right|\leq
1+\frac{|z|}{\textup{dist}\,(z,\mathscr{S}^{+}_{\vartheta})},\hspace{1.5ex}
z\in\mathscr{G}^{+}_{\vartheta},\ \ 0\leq\xi<\infty\,,\varepsilon>0\,.
\end{equation}
For \(\vartheta:\,0<\vartheta<\frac{\pi}{4}\), and \(\varepsilon>0\), let us define
\begin{equation}
\label{fte+}
f_{\varepsilon,\vartheta}^{+}(z)=
\frac{e^{i\vartheta}}{2\sqrt{\pi}}
\int\limits_{0}^{+\infty}
e^{-\xi^2\textstyle{e}^{\scriptscriptstyle{2i\vartheta}}/4}
\frac{1}{1-ze^{i\varepsilon\xi{}\textstyle{e}^{\scriptscriptstyle{i\vartheta}}}}\,d\xi, \hspace{1.5ex}
\,z\in\mathscr{G}_{\vartheta}^{+}\,.
\end{equation}
In view of the equality
\begin{equation}
\label{Moe}
\Big|e^{-\xi^2\textstyle{e}^{\scriptscriptstyle{2i\vartheta}}/4}\Big|=
e^{-\xi^2\cos2\theta},\ \ 0\leq\xi<\infty\,,
\end{equation}
and the estimates \eqref{fet+}, the function \(e^{-\xi^2\textstyle{e}^{\scriptscriptstyle{2i\vartheta}}/4}
\dfrac{1}{1-ze^{i\varepsilon\xi{}\textstyle{e}^{\scriptscriptstyle{i\vartheta}}}}\)
which appears under the integral \eqref{fte+} admits the estimate
\begin{multline}
\label{Coes}
\left|e^{-\xi^2\textstyle{e}^{\scriptscriptstyle{2i\vartheta}}/4}
\dfrac{1}{1-ze^{i\varepsilon\xi{}\textstyle{e}^{\scriptscriptstyle{i\vartheta}}}}\right|
\leq\left(1+\frac{|z|}{\textup{dist}\,(z,\mathscr{S}^{+}_{\vartheta})}\right)
\cdot{}e^{-\xi^2\cos2\theta}
,\\
  0\leq\xi<\infty\,,\ \ \varepsilon>0\,.
\end{multline}
Since the function \(e^{-\xi^2\cos2\theta}\) is integrable:
\begin{equation}
\label{CIn}
\int\limits_{0}^{\infty}e^{-\xi^2\cos2\theta}\,d\xi=\sqrt{\tfrac{\pi}{\cos{2\vartheta}}}\,,
\end{equation}
the integral in \eqref{fte+} exists. \emph{The function \(f_{\varepsilon,\vartheta}^{+}(z)\)}
which is determined by means of this integral \emph{is well defined and holomorphic for \(z\in\mathscr{G}_{\vartheta}^{+}\)}. In view of \eqref{SpRo+},
\begin{equation}
\label{CoIn}
f_{\varepsilon}^{+}(z)=f_{\varepsilon,\vartheta}^{+}(z) \ \ \textup{for} \ \ z\in\mathbb{D}\,.
\end{equation}
Thus \emph{the function \(f_{\varepsilon,\vartheta}^{+}\) is an analytic continuation of the function \(f_{\varepsilon}^{+}\) from the unit disc \(\mathbb{D}\) to the domain
\(\mathscr{G}_{\vartheta}^{+}\).}

From \eqref{fte+}, \eqref{Coes}, \eqref{CIn} we conclude that the family
\(\lbrace{}f_{\varepsilon,\vartheta}^{+}\rbrace_{\varepsilon>0}\) is
locally bounded in the domain \(\mathscr{G}_{\vartheta}^{+}\), where the bound is uniform
with respect to \(\varepsilon\):
\begin{equation}
\label{LUEs}
\big|\lbrace{}f_{\varepsilon,\vartheta}^{+}(z)\big|\leq\sqrt{\tfrac{1}{\cos2\vartheta}}
\cdot\left(1+\frac{|z|}{\textup{dist}\,(z,\mathscr{S}^{+}_{\vartheta})}\right)\,, \ \
z\in\mathscr{G}_{\vartheta}^{+}\,,\ \varepsilon>0.
\end{equation}
In particular, the family \(\lbrace{}f_{\varepsilon,\vartheta}^{+}(z)\rbrace_{\varepsilon>0}\) is normal in
the domain \(\mathscr{G}_{\vartheta}^{+}\).
According to \eqref{CoIn}, the limiting relation \eqref{BLRa} can be interpreted as
\[\lim\limits_{\varepsilon\to0}f_{\varepsilon,\vartheta}^{+}(z)=\frac{1}{2(1-z)},\
\ z\in\mathbb{D}\,.\]
From this and from the normality of the family \(\lbrace{}f_{\varepsilon,\vartheta}^{+}(z)\rbrace_{\varepsilon>0}\) in \(\mathscr{G}_{\vartheta}^{+}\) it follows that
\begin{equation}
\label{ELR}
\lim\limits_{\varepsilon\to0}f_{\varepsilon,\vartheta}^{+}(z)=\frac{1}{2(1-z)}
\ \ \textup{for} \ \ z\in\mathscr{G}_{\vartheta}^{+} \ \ \textup{locally uniformly}\,.
\end{equation}
The relation \eqref{ELR} can be also obtained from \eqref{fte+} and the Lebesgue
dominated convergence theorem.

Let us summarize the above-stated as
\begin{lemma}
\label{L2.1}For all numbers \(\vartheta\) and \(\varepsilon\), \(0<\vartheta<\pi/4\),
 \(0<\varepsilon\), there exists a
function \(f^{+}_{\varepsilon,\vartheta}\,(\,.\,)\) which possess the properties:
\begin{enumerate}
\item The function \(f^{+}_{\varepsilon,\vartheta}\,(\,.\,)\) is holomorphic in the domain \(\mathscr{G}^{+}_{\vartheta}\) and satisfies the estimate \eqref{LUEs} there.
\item The
    function \(f^{+}_{\varepsilon,\vartheta}\,(\,.\,)\) is an analytic continuation of the
    function \(f^{+}_{\varepsilon}(\,.\,)\), \eqref{SpIna}, from the unit disc \(\mathbb{D}\) to the domain \(\mathscr{G}^{+}_{\vartheta}\), i.e. the equality \eqref{CoIn} holds.
\item The limiting relation \eqref{ELR} holds.
\end{enumerate}
\end{lemma}

If \(\vartheta=\pi/4\), the function \(f^{+}_{\varepsilon,\pi/4}\) can not be defined by   the integral \eqref{fte+} with \(\theta=\pi/4\). This integral does not converges absolutely.%
 \footnote{We still
 can assign a meaning to the integral \eqref{fte+} (with \(\vartheta=\pi/4\)) by some regularization method. For example we can
 consider this integral as an improper integral. However even if we define the function
 \(f^{+}_{\varepsilon,\pi/4}\)  by an improper integral,
 we would be unable to prove the limiting relation \eqref{ELR} starting from such a
 definition.} To define the function \(f^{+}_{\varepsilon,\pi/4}\) in the domain
\(\mathscr{G}^{+}_{\pi/4}\), we glue together the functions \(\lbrace{}f^{+}_{\varepsilon,\vartheta}\rbrace_{0<\vartheta<\pi/4}\) into a single function.

\begin{remark}
If \(\vartheta^{\prime},\vartheta^{\prime\prime}\!\in\,\,]0,\pi/4[\,\), the functions \(f^{+}_{\varepsilon,\theta^{\prime}}\)
and \(f^{+}_{\varepsilon,\theta^{\prime\prime}}\) are defined and holomorphic in the domains
\(\mathscr{G}^{+}_{\vartheta^{\prime}}\) and \(\mathscr{G}^{+}_{\vartheta^{\prime\prime}}\) respectively. The unit disc \(\mathbb{D}\) is contained in the intersection of these domains:
\(\mathbb{D}\subset\mathscr{G}^{+}_{\vartheta^{\prime}}
\bigcap\mathscr{G}^{+}_{\vartheta^{\prime\prime}}\). The functions \(f^{+}_{\varepsilon,\theta^{\prime}}\) and \(f^{+}_{\varepsilon,\theta^{\prime\prime}}\)
coincide on \(D\):
\(f^{+}_{\varepsilon,\theta^{\prime}}(z)=f^{+}_{\varepsilon,\theta^{\prime\prime}}\,(=
f^{+}_{\varepsilon}(z))\) for \(z\in\mathbb{D}\). However if
\(\theta^{\prime}\not=\theta^{\prime\prime}\), then \textsf{the intersection
\(\mathscr{G}^{+}_{\vartheta^{\prime}}
\bigcap\mathscr{G}^{+}_{\vartheta^{\prime\prime}}\) is not connected.}
Therefore we can only conclude  that the functions  \(f^{+}_{\varepsilon,\theta^{\prime}}\) and \(f^{+}_{\varepsilon,\theta^{\prime\prime}}\)
coincide on the connected component
of the open set \(\mathscr{G}^{+}_{\vartheta^{\prime}}
\bigcap\mathscr{G}^{+}_{\vartheta^{\prime\prime}}\) that contains the origin. This circumstance complicates the reasoning a little.  To carry out the reasoning smoothly, we introduce
an auxiliary monotonic sequence of \textsf{connected} open sets \(O_n\) with compact closures \(\overline{O_n}\) which sequence
\textsf{exhausts} the domain \(\mathscr{G}^{+}_{\pi/4}\).
\end{remark}

A simple geometric construction\footnote{Such a construction can be done in many different ways. We omit a formal geometric construction of the sequence \(\lbrace{}O_n\rbrace_{1\leq{}n<\infty}\).} shows that
the domain \(\mathscr{G}^{+}=\mathscr{G}_{\pi/4}^{+}\) can be represented
as
\begin{equation}
\label{Repr}
\mathscr{G}_{\pi/4}^{+}=\bigcup_{1\leq{}n<\infty}O_n,
\end{equation}
where the sequence \(\lbrace{}O_n\rbrace_{1\leq{}n<\infty}\) satisfies the conditions:
\begin{enumerate}
\item
Each \(O_n\) is an open set;
\item
The closure \(\overline{O_n}\) of the set \(O_n\) is a compact set which is
contained in the domain \(\mathscr{G}_{\pi/4}^{+}\):
{\ }\\[-4.0ex]
\begin{multline}
\label{PropOn}
\hfill
\overline{O_n} \ \textup{is a compact set}, \ \overline{O_n}\subset\mathscr{G}_{\pi/4}^{+};
\hfill
\end{multline}
\item Each \(O_n\) is a connected set.
\item
 The sequence \(\lbrace{}O_n\rbrace_{1\leq{}n<\infty}\) increases:
{\ }\\[-4.0ex]
\begin{multline}
\label{Increa}
\hfill
O_1\subseteq{}O_2\subseteq{}O_3\,\ldots\,;
\hfill
\end{multline}{\ }\\[-6.0ex]
\item Every set \(O_n\), \(n=1,\,2,\,3\ldots\)\,, contains the disc \(\mathbb{D}_{1/2}\),
\,\(\mathbb{D}_{1/2}=\lbrace\zeta\in\mathbb{C}:\,|\zeta|<1/2\rbrace\,:\)
{\ }\\[-4.0ex]
\begin{multline}
\label{Comm}
\hfill
\mathbb{D}_{1/2}\subset{}O_n,\,\,n=1,\,2,\,3,\,\ldots\,.\hfill
\end{multline}
\end{enumerate}

Let us choose and fix such a sequence of sets \(\lbrace{}O_n\rbrace_{1\leq{}n<\infty}\).

For \(\theta\in\,]0,\pi/2[\)\,, the domain \(\mathscr{G}_{\vartheta}^{+}\) depends on
\(\vartheta\) continuously, where the convergence of domains is the
kernel convergence in the sense of Caratheodory. (Regarding the notion of kernel convergence, we refer to \cite[section 1.4]{Pom}.)
The relation \(\lim\limits_{\vartheta\to\pi/4-0}\mathscr{G}_{\vartheta}=\mathscr{G}_{\pi/4}\)
means, in particular, that for every compact set \(K\),  \(K\in\mathscr{G}_{\pi/4}\),
there exists \(\vartheta_K\), \(0<\vartheta_K<\pi/4\), such that \(K\subset\mathscr{G}_{\vartheta}\) for \(\vartheta:\,\vartheta_K\leq\vartheta<\pi/4\).

By choosing  the set \(\overline{O_n}\) as \(K\), \eqref{PropOn}, we conclude that there exists \(\vartheta_n\), \(0<\theta_n<\pi/4\), such that
\begin{equation}
\label{Tetn}
\overline{O}_n\subset\mathscr{G}^{+}_{\vartheta_n}\,.
\end{equation}

As  was stated in Lemma \ref{L2.1}, the function \(f^{+}_{\varepsilon,\theta_n}\)
is well defined and holomorphic in the domain \(\mathscr{G}^{+}_{\vartheta_n}\).
In particular, the function \(f^{+}_{\varepsilon,\theta_n}\) is well defined and holomorphic in the domain \(O_n\).

If \(n_1<n_2\), then the functions \(f^{+}_{\varepsilon,\theta_{n_1}}\)
and \(f^{+}_{\varepsilon,\theta_{n_2}}\) are defined and holomorphic in the domains
\(\mathscr{G}^{+}_{\vartheta_{n_1}}\) and \(\mathscr{G}^{+}_{\vartheta_{n_2}}\) respectively.
Since
\begin{equation*}
f^{+}_{\varepsilon,\theta_{n_1}}(z)=f^{+}_{\varepsilon,\theta_{n_2}}(z) \
(\,=f^{+}_{\varepsilon}(z)\,) \ \text{ for } \  z\in\mathbb{D}_{1/2}
\end{equation*}
and \(\mathbb{D}_{1/2}\subset{}O_{n_1}\subset{}O_{n_2}\),
we conclude that
\begin{equation}
\label{Her}
f^{+}_{\varepsilon,\theta_{n_1}}(z)=f^{+}_{\varepsilon,\theta_{n_2}}(z) \
\text{ for } \ z\in{}O_{n_1}, \ \ n_1<n_2\,.
\end{equation}
In view of \eqref{Repr}, \eqref{Increa} and \eqref{Her}, the sequence of functions
\(\lbrace{}f^{+}_{\varepsilon,\theta_n}\rbrace_{1\leq{}n<\infty}\)
    can be glued together into a single function, which is defined
    on the set \(\bigcup\limits_{n}O_n=\mathscr{G}^{+}_{\pi/4}\). We denote this function by \(f^{+}_{\varepsilon,\pi/4}\):
\begin{equation}
\label{Restr}
f^{+}_{\varepsilon,\pi/4}(z)=f^{+}_{\varepsilon,\vartheta_{n}}(z) \ \text{ for } \ z\in{}O_n\,,\
\ n=1,\,2,\,3,\,\ldots\,.
\end{equation}
{\ }\\
The equalities \eqref{Repr},\eqref{Restr} serve as a \emph{definition} of the function \(f^{+}_{\varepsilon,\pi/4}\) in the domain \(\mathscr{G}^{+}_{\pi/4}\). By virtue of \eqref{Her}, this definition is non-contradictory.

\begin{lemma}
\label{L22}
For every \(\varepsilon,\,\varepsilon>0\), there exists a function \(f^{+}_{\varepsilon,\,\pi/4}\) such that
\begin{enumerate}
\item
The function \(f^{+}_{\varepsilon,\,\pi/4}\,(\,.\,)\) is holomorphic in the domain \(\mathscr{G}^{+}_{\varepsilon/4}\).
\item
The function \(f^{+}_{\varepsilon,\pi/4}\,(\,.\,)\) is an analytic continuation of the
function \(f^{+}_{\varepsilon}(\,.\,)\), which was defined by \eqref{SpIna}, from the unit disc \(\mathbb{D}\) to the domain \(\mathscr{G}^{+}_{\pi/4}\), i.e. the equality
\begin{equation}
\label{AnCoP}
f^{+}_{\varepsilon,\pi/4}\,(z)=f^{+}_{\varepsilon}(z) \ \
\end{equation}
holds for every \(z\in\mathbb{D}\).
\item
The functional family \(\lbrace{}f^{+}_{\varepsilon,\pi/4}\rbrace_{\,0<\varepsilon<\infty}\)
is locally bounded in \(\mathscr{G}^{+}_{\pi/4}\). In other words,
for each compact set \(K\),
\(K\subset\mathscr{G}^{+}_{\pi/4}\), the estimate
\begin{equation}
\label{LUEsP}
\big|f^{+}_{\varepsilon,\pi/4}(z)\big|\leq{}C_K^{+}, \ \ \forall z\in{}K,
\end{equation}
holds, where the value \(C_K^{+}<\infty\) does not depend on \(\varepsilon\).
\item
The limiting relation
\begin{equation}
\label{LiRePl}
\lim_{\varepsilon\to+0}f^{+}_{\varepsilon,\pi/4}(z)=\frac{1}{2(1-z)},
\ \ \forall{}z\in\mathscr{G}^{+}_{\pi/4},
\end{equation}
holds.  In \eqref{LiRePl}, the limit is locally uniform with respect to \(z\in\mathscr{G}^{+}_{\pi/4}\).
\end{enumerate}
\end{lemma}
\begin{proof}
\textsf{1}. Taking into account \eqref{Restr}, \eqref{Tetn} and the holomorphy
of the function \(f^{+}_{\varepsilon,\vartheta_n}\) on the domain
\(\mathscr{G}^{+}_{\vartheta_n}\), we conclude that the function
\(f^{+}_{\varepsilon,\pi/4}\) is holomorphic on the domain \(O_n\). In view of
\eqref{Repr}, the function \(f^{+}_{\varepsilon,\pi/4}\) is holomorphic on
the domain \(\mathscr{G}^{+}_{\pi/4}\).\\
\textsf{2}. By virtue of \eqref{Restr}, \eqref{Comm} and \eqref{CoIn},
the equality \eqref{AnCoP} holds for every \(z\in\mathbb{D}_{1/2}\).
Since both functions \(f^{+}_{\varepsilon,\pi/4}\) and \(f^{+}_{\varepsilon}\),
are holomorphic on \(\mathbb{D}\), the equality \eqref{AnCoP} holds for every \(z\in\mathbb{D}\).\\
\textsf{3}. Let \(K\) be a compact set, \(K\subset\mathscr{G}^{+}_{\pi/4}\).
In view of \eqref{Repr}, \(K\subset\bigcup\limits_{1\leq{}n<\infty}{O_{n}}\).
In view of \eqref{Increa}, \(K\subset{}O_{n_0}\) for some \(n_0\).
Furthemore, \(K\subset\mathscr{G}^{+}_{\vartheta_{n_0}}\), \eqref{Tetn}.
In particular,
\begin{equation}
\label{PosDi}
\text{dist}(K,\mathscr{S}^{+}_{\vartheta_{n_0}})>0\,,
\end{equation}
where \(\text{dist}(K,\mathscr{S}^{+}_{\vartheta_{n_0}})\) is the distance
from the set \(K\) to the boundary \(\mathscr{S}^{+}_{\vartheta_{n_0}}\) of the
domain \(\mathscr{G}^{+}_{\vartheta_{n_0}}\).
The equality \eqref{Restr} (with \(n=n_0\)) and the estimate \eqref{LUEs}
imply the estimate \eqref{LUEsP} with
\footnotesize
\begin{equation*}
C_{K}=\sqrt{\frac{1}{\cos2\vartheta_{\!n_0}}}
\cdot\left(1+\frac{\max\limits_{\zeta\in{}K}|\zeta|}
{\textup{dist}\,(K,\mathscr{S}^{+}_{\vartheta_{\!n_0}})}\right).
\end{equation*}
\normalsize
\textsf{4}. Let \(K\) be a compact set, \(K\subset\mathscr{G}^{+}_{\pi/4}\).
As we saw, there exists \(n_0\) such that \(K\subset{}O_{n_0}\subset\mathscr{G}^{+}_{\vartheta_{n_0}}\). According to
Lemma \ref{L2.1}, \(\lim\limits_{\varepsilon\to+0}f^{+}_{\varepsilon,\vartheta_{n_0}}(z)=
\dfrac{1}{2(1-z)}\) for each \(z\in{}K\). Moreover, this limiting relation holds uniformly with respect to \(z\in{}K\). In view of \eqref{Restr},
\(f^{+}_{\varepsilon,\pi/4}(z)=f^{+}_{\varepsilon,\vartheta_{n_0}}(z)\) for \(z\in{}K\).
\end{proof}

The same reasoning can be carried out for the function \(f^{-}_{\varepsilon}\).

\begin{lemma}
\label{L23}
For every \(\varepsilon,\,\varepsilon>0\), there exists a function \(f^{-}_{\varepsilon,\,\pi/4}\) such that
\begin{enumerate}
\item
The function \(f^{-}_{\varepsilon,\,\pi/4}\,(\,.\,)\) is holomorphic in the domain \(\mathscr{G}^{-}_{\varepsilon/4}\).
\item
The function \(f^{-}_{\varepsilon,\pi/4}\,(\,.\,)\) is an analytic continuation of the
function \(f^{-}_{\varepsilon}(\,.\,)\), which was defined by \eqref{SpInb}, from the unit disc \(\mathbb{D}\) to the domain \(\mathscr{G}^{-}_{\pi/4}\), i.e. the equality
\begin{equation}
\label{AnCoM}
f^{-}_{\varepsilon,\pi/4}\,(z)=f^{-}_{\varepsilon}(z) \ \
\end{equation}
holds for every \(z\in\mathbb{D}\).
\item
The functional family \(\lbrace{}f^{-}_{\varepsilon,\pi/4}\rbrace_{\,0<\varepsilon<\infty}\)
is locally bounded in \(\mathscr{G}^{-}_{\pi/4}\). In other words,
for each compact set \(K\),
\(K\subset\mathscr{G}^{-}_{\pi/4}\), the estimate
\begin{equation}
\label{LUEsM}
\big|f^{-}_{\varepsilon,\pi/4}(z)\big|\leq{}C_K^{-}, \ \ \forall z\in{}K,
\end{equation}
holds, where the value \(C_K^{-}<\infty\) does not depend on \(\varepsilon\).
\item
The limiting relation
\begin{equation}
\label{LiReMin}
\lim_{\varepsilon\to+0}f^{-}_{\varepsilon,\pi/4}(z)=\frac{1}{2(1-z)},
\ \ \forall{}z\in\mathscr{G}^{-}_{\pi/4},
\end{equation}
holds.  In \eqref{LiReMin}, the limit is locally uniform with respect to \(z\in\mathscr{G}^{-}_{\pi/4}\).
\end{enumerate}
\end{lemma}

We denote by \(\mathscr{G}\) the connected component of the open set \(\mathscr{G}^{+}_{\pi/4}\bigcap\mathscr{G}^{-}_{\pi/4}\), which contains the origin.
The set \(\mathscr{G}\) admits an explicit description.\\[1.0ex]
\begin{definition}
\label{DeDo}
Let \(\mathscr{C}\) be a closed curve whose parametric representation is
\begin{equation}
\label{HeShCu}
\mathscr{C}=\lbrace{}\zeta\in\mathbb{C}:\,\zeta=e^{|t|+it}\!\!,\, \ \text{where \(t\) runs over} \ [-\pi,+\pi]\,\,\rbrace.
\end{equation}
The domain \(\mathscr{G}\) is the interior of the curve \(\mathscr{C}\).
\end{definition}
\begin{remark}
\label{WTMr}
It is worth mentioning that the domain \(\mathscr{G}\) contains the open unit disc, more precisely\\[-4.0ex]
{\ }\vspace*{-1.0ex}
\begin{equation}
\label{WTM}
\overline{\mathbb{D}}\setminus\{1\}\subset\mathscr{G}\,.
\end{equation}
\end{remark}

{\ }\\
The heart-shaped curve \(\mathscr{C}\) is the outer Jordan curve in Figure 1. It plotted by
a solid blue line. The unit circle \(\mathbb{T}\) is plotted by a solid red line. The curve \(\mathscr{C}\) intersects
the real axis at the points with coordinates \((1,0)\) and (\(-e^{\pi},0)\).
(\(e^{\pi}= 23.140692632779267\ldots\,.\))
The fragments of this
curve near these points are plotted in Figures 2 and 3.

\noindent
\hspace*{1.5ex}
\begin{minipage}{0.26\linewidth}
\begin{center}{\includegraphics*[height=1.2\linewidth,width=1.2\linewidth]{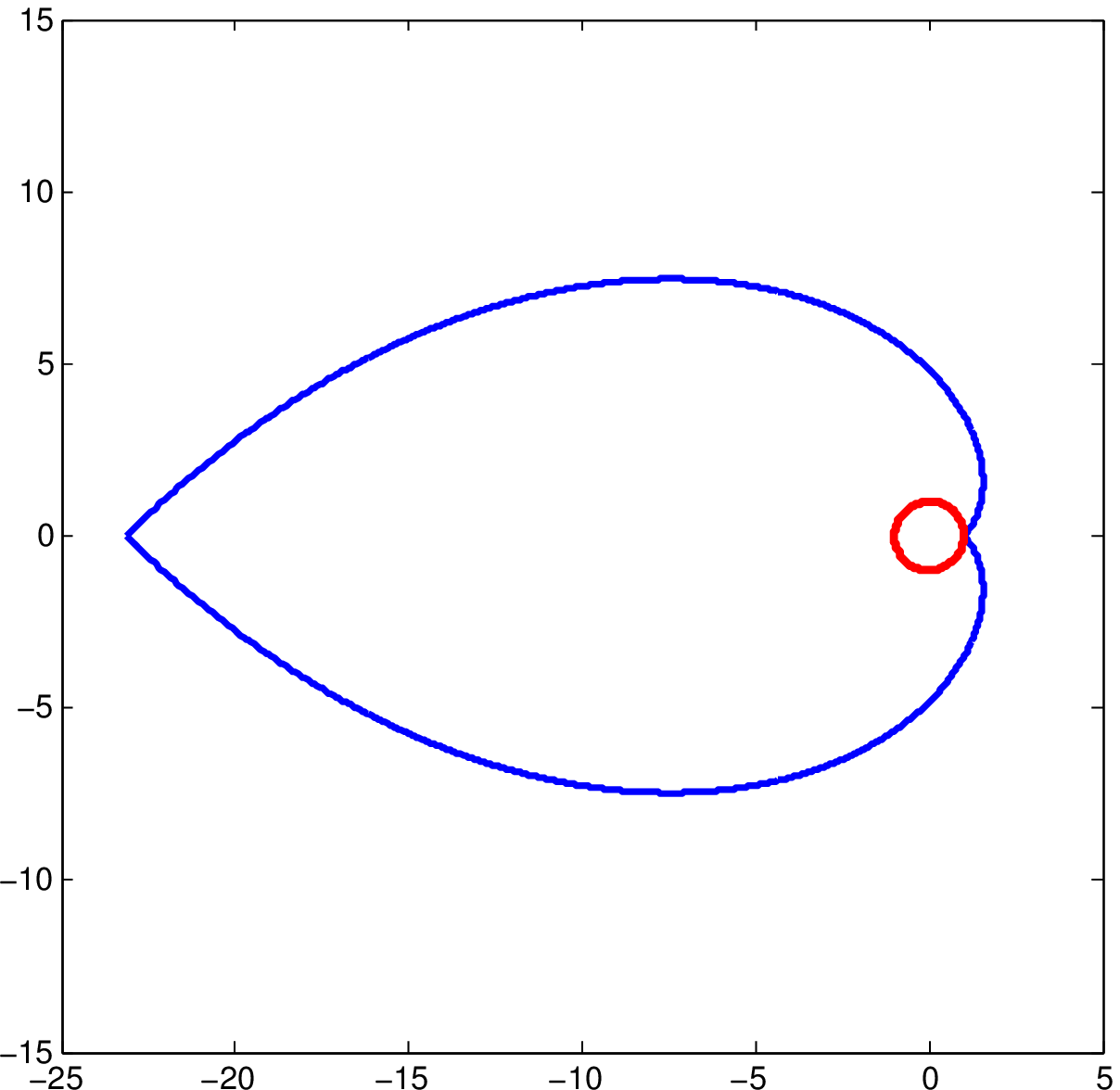}}
\end{center}
\end{minipage}
\hspace*{3.3ex}
\begin{minipage}{0.26\linewidth}
\vspace*{0.9ex}
\begin{center}{\includegraphics*[height=1.2\linewidth,width=1.2\linewidth]{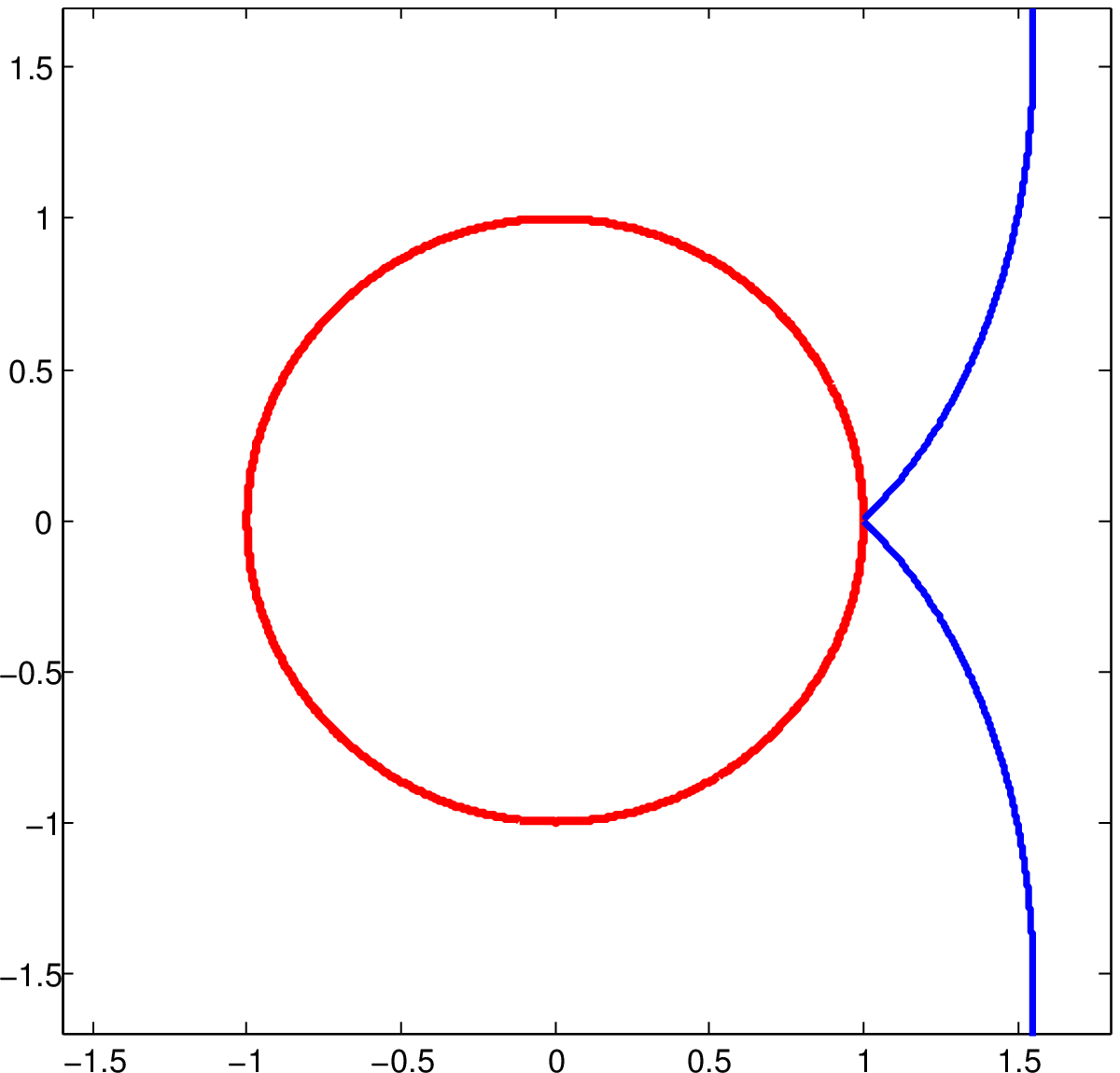}}
\end{center}
\end{minipage}
\hspace*{3.3ex}
\begin{minipage}{0.26\linewidth}
\vspace*{0.80ex}
\begin{center}{\includegraphics*[height=1.2\linewidth,width=1.2\linewidth]{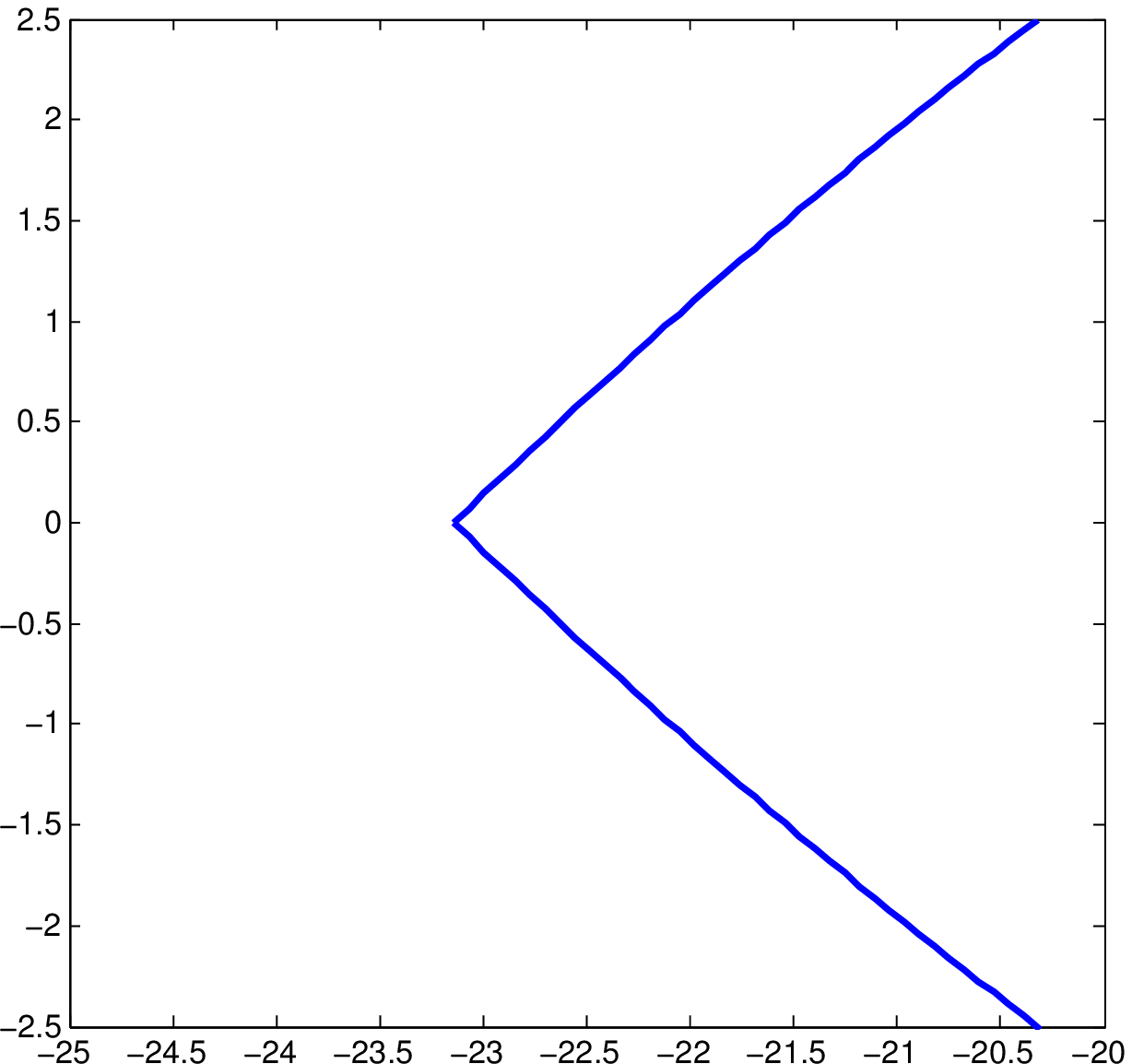}}
\end{center}
\end{minipage}
\vspace{1.5ex}

\centerline{\hspace*{0.0ex} Figure 1 \hspace*{15.0ex}  Figure 2 \hspace*{15.0ex}Figure 3 \hspace*{0.0ex}}

\vspace{2.0ex}
\noindent
In the domain \(\mathscr{G}\), both functions \(f^{+}_{\varepsilon,\pi/4}\) and \(f^{-}_{\varepsilon,\pi/4}\)
are defined and holomorphic. Hence, the sum
\begin{math}
f^{+}_{\varepsilon,\pi/4}(z)+f^{-}_{\varepsilon,\pi/4}(z)
\end{math}
of these functions is defined and holomorphic on \(\mathscr{G}\). According to \eqref{AnCoP}, \eqref{AnCoM}, and \eqref{SpI},
\begin{equation*}
f^{+}_{\varepsilon,\pi/4}(z)+f^{-}_{\varepsilon,\pi/4}(z)=f_{\varepsilon}(z),\ \forall\,z\in\mathbb{D}\,,
\end{equation*}
where \(f_{\varepsilon}\) is  defined by \eqref{TeSRb}. The function \(f^{+}_{\varepsilon,\pi/4}+f^{-}_{\varepsilon,\pi/4}\)
is holomorphic in \(\mathscr{G}\), the function \(f_{\varepsilon}(z)\) is an entire function.
Therefore
\begin{equation}
\label{MaCo}
f^{+}_{\varepsilon,\pi/4}(z)+f^{-}_{\varepsilon,\pi/4}(z)=f_{\varepsilon}(z),\ \forall\,z\in\mathscr{G}\,.
\end{equation}
The following statement follows from the last equality and from the properties\footnote{\,These properties were summarized in Lemmas \ref{L22}, \ref{L23}.} of the functional families
\(\lbrace{}f^{\pm}_{\varepsilon,\pi/4}\rbrace_{0<\varepsilon<\infty}\)\,.
\begin{theorem}
\label{ConTh}
For \(\varepsilon>0\), let \(f_{\varepsilon}(z)\) be the function which is defined as the
sum of the power series \eqref{TeSRb}. Let \(\mathscr{G}\) be the domain
 introduced in Definition \ref{DeDo}. \\
 Then
\begin{enumerate}
\item
For each \(\varepsilon>0\), the function \(f_{\varepsilon}\) is an entire function of order zero.
\item
The functional family \(\lbrace{}f_{\varepsilon}\rbrace\) is locally bounded in \(\mathscr{G}\).
This means that for every compact subset \(K\) of the domain \(\mathscr{G}\),
 \(K\Subset{}G,\) the inequality
 \begin{equation}
 \label{UnBo}
 |f_{\varepsilon}(z)|\leq{}C_K, \ \ \forall z\in{}K\,,\ \ \forall\varepsilon>0\,,
 \end{equation}
 holds, where \(C_K<\infty\) is a constant which does not depend on \(z\) and \(\varepsilon\).
 \item
 The limiting relation holds
 \begin{equation}
 \label{CoSe}
 \lim_{\varepsilon\to+0}f_{\varepsilon}(z)=\frac{1}{1-z}\,, \ \ \text{for each} \ \ z\in\mathscr{G}.
 \end{equation}
 The convergence \(f_{\varepsilon}(z)\) to \(\frac{1}{1-z}\) is locally uniform with respect to \(z\in\mathscr{G}\).
\end{enumerate}
\end{theorem}
\begin{remark}
\label{Pri}
The proof of \textup{Theorem \ref{ConTh}} is close to the proof of the Theorem of Le Roy and
Lindel\"of, see \cite[page 340]{Die}.
Using the method of Le Roy and Lindel\"of, D.\,Khavinson gave a short proof of the following result of Tao Qian.

Let \(f(z)=\sum\limits_{1\leq n<\infty}g(n)z^n\), where \(g(\zeta)\) is a function           holomorphic and bounded in the sector \(S_{\psi}=\{z:|\textup{arg}\,z|<\psi\}\),
where \(0<\psi\leq\pi/2\). Then the function \(f(z)\), which is, of course, analytic
in the unit disc \(\mathbb{D}\), extends in the heart-shaped domain
\(G_{\psi}:=\{z=re^{i\vartheta}:2\pi-\cos\psi\,\ln{}r>\vartheta>\cos\psi\,\ln{}r\}\).

If \(\psi=\pi/4\), then \(G_{\psi}=\mathscr{G}\), where the domain \(\mathscr{G}\) was
introduced in Definition~\ref{DeDo}. For each \(\varepsilon>0\), the entire function
\(g_{\varepsilon}(\zeta):=e^{-\varepsilon\zeta^2}\) is bounded in the sector \(S_{\pi/4}\).
From the reasoning of the paper \cite{Kh} it follows that the family of functions \(f_{\varepsilon}(z)=\sum\limits_{1\leq n<\infty}g_{\varepsilon}(n)z^n\) is locally
bounded in \(\mathscr{G}\) uniformly with respect to \(\varepsilon>0\). Our Theorem~\ref{ConTh} can be derived from this.
\end{remark}
\begin{remark}
\label{Fru}
The paper \cite{Fr} also is related to the paper \cite{Kh} and to our Theorem~\ref{ConTh}.
See Theorem 3.8 of \cite{Fr}. This Theorem formally contains our Theorem \ref{ConTh}
and in fact is based on the method of Le Roy and Lindel\"of. In \cite{Fr}, the terminology
of the non-standard analysis is used. In \cite{Fr}, the exposition of Theorem 3.8 is not quite clean,
but it can be cleaned. Theorem 3.9 in \cite{Fr} is wrong.
\end{remark}
\section{Theta summation method. Divergence.}
\label{Sec3}
\setcounter{equation}{0}
\begin{theorem}\hspace*{1.5ex}
\label{DivTh}
Let the function \(f_{\varepsilon}\) and the domain \(\mathscr{G}\) be the same that in Theorem \ref{ConTh}.

 Then for every \(z\not\in\mathscr{G}\) the limiting relation
\begin{equation}
\label{DivimRel}
\varlimsup_{\varepsilon\to+0}|f_{\varepsilon}(z)|=\infty
 \end{equation}
 holds.
\end{theorem}{\ }\\

\noindent
It is clear that
\begin{equation}
\label{EC}
\lim_{\varepsilon\to+0}\sum\limits_{-\infty<n<0}e^{-\varepsilon{}n^2}z^n=
\sum\limits_{-\infty<n<0}z^n=\frac{z}{z-1}\,\ccomma \ \ \forall{}z\in\mathbb{D}^{-}.
\end{equation}
\label{CA}
For \(\varepsilon>0\) and \(z\in\mathbb{C}\setminus\{0\}\), let
\begin{equation}
\label{AuH}
h_{\varepsilon}(z)=\sum\limits_{-\infty<n<\infty}e^{-\varepsilon{}n^2}z^n
\end{equation}
For each \(\varepsilon>0\), the function \(h_{\varepsilon}(\,.\,)\) is holomorphic in the domain
\(\mathbb{C}\setminus\{0\}\) and satisfies the relation
\begin{equation}
\label{SymRel}
h_{\varepsilon}(z)=h_{\varepsilon}(z^{-1})\,.
\end{equation}

\begin{definition}
\label{DivvS}
Let
\begin{subequations}
\label{Divv}
\begin{align}
\label{Divva}
\mathscr{V}_{f}=\{z\in\mathbb{C}:\,\varlimsup_{\varepsilon\to+0}|f_{\varepsilon}(z)|&=\infty\},\\
\label{Divvb}
\mathscr{V}_{h}=\{z\in\mathbb{C}:\,\varlimsup_{\varepsilon\to+0}|h_{\varepsilon}(z)|&=\infty\},
\end{align}
\end{subequations}
The sets \(\mathscr{V}_{f}\) and \(\mathscr{V}_{h}\) are said to be \emph{the divergence set for the functional family
\(\{f_{\varepsilon}\}_{\varepsilon>0}\)} and \emph{the divergence set for the functional family
\(\{h_{\varepsilon}\}_{\varepsilon>0}\)} respectively.\\
\end{definition}

\begin{remark}
\label{EqFo}
Theorem \ref{DivTh} is equivalent to the relation
\begin{equation}
\label{EqRel}
\mathbb{C}\setminus\mathscr{G}=\mathscr{V}_{f}\,.
\end{equation}
\end{remark}

\begin{remark}
\label{Inva}
In view of \eqref{SymRel}, the set \(\mathscr{V}_{h}\) is invariant with respect to the transformation
\(z\to{}z^{-1}\): \, \(z\in\mathscr{V}_{h}\) if and only \(z^{-1}\in\mathscr{V}_{h}\).\\
\end{remark}

\noindent
In view of  \eqref{CoSe} and \eqref{WTM},
\begin{equation}
\label{Incl}
\mathscr{V}_{f}\subset\{1\}\cup\mathbb{D}^{-}\,.
\end{equation}
It is clear that
\begin{equation}
\label{Int}
\{1\}\subset\mathscr{V}_{f}\cap\mathscr{V}_{h}\,.
\end{equation}

\noindent
Because of \eqref{Incl}, \eqref{Int}, the equality
\begin{equation*}
h_{\varepsilon}(z)=f_{\varepsilon}(z)+
\sum\limits_{-\infty<n<0}e^{-\varepsilon{}n^2}z^n,\ \ \forall\,z\in\mathbb{C},\ \forall\,\varepsilon>0\,,
\end{equation*}
and \eqref{EC}, the following statement is evident.\\

\begin{lemma}
\label{L3.1}
For the divergence sets \(\mathscr{V}_{f}\) and \(\mathscr{V}_{h}\), the following relation holds:
\begin{equation}
\label{SeDe}
\mathscr{V}_{f}=\{1\}\cup\big(\mathscr{V}_{h}\cap\mathbb{D}^{-}\big)\,.
\end{equation}
\end{lemma}{\ }\\

\noindent
To find the divergence set \(\mathscr{V}_{h}\), we use the equality
\begin{equation}
\label{PuTr}
h_{\varepsilon}(z)=H_{\varepsilon}\big(\tfrac{\ln{}z}{2i}\big), \ \ \forall{}z\in\mathbb{C}\setminus\{0\}
\end{equation}
where
\begin{equation}
\label{PuTrF}
H_{\varepsilon}(\zeta)=\sqrt{\tfrac{\pi}{\varepsilon}}\!\!\!\!
\sum\limits_{-\infty<n<\infty}\!\!\text{\large\(e\)}^{-(\zeta-n\pi)^{2}/\varepsilon},\ \
\varepsilon>0,\ \, \zeta\in\mathbb{C}\,.
\end{equation}
(The converges of the series on the right hand side of \eqref{PuTrF} is explained in Corollary
\ref{CorCon} below.)
The equality \eqref{PuTr}, where the functions \,\(h_{\varepsilon}(\,.\,)\) \,and \, \(H_{\varepsilon}(\,.\,)\) \,are defined by \eqref{AuH} and \eqref{PuTrF}, is a particular case
of the Jacobi transformation for the theta function \(\vartheta_3(z|\tau)\). (See
\cite[Chapter XXI, sect.\,21.51]{WhWa}, the formula before Example 1 there.)
However the equality \eqref{PuTr} can be obtained without any reference to the theory
of theta functions, but by using the Poisson summation formula. (Concerning the Poisson summation
formula we refer to \cite[Chapter 2, sect.\,7.5]{DyMcK}, pages 111-112.)

The function \(H_{\varepsilon}(\zeta)\) is a periodic function with respect to \(\zeta\) with a period \(\pi\)
and also is an even function:
\begin{equation}
\label{SySeH}
H_{\varepsilon}(\zeta+\pi)\equiv{}H_{\varepsilon}(\zeta)\,, \ \
H_{\varepsilon}(\zeta)\equiv{}H_{\varepsilon}(-\zeta)\,.
\end{equation}
Therefore the function \(H_{\varepsilon}\big(\tfrac{\ln{}z}{2i})\) is a single-valued function of \(z\).
In particular, this function does not depend on a choice of branch of \(\ln z\).

\begin{definition}
\label{DHS}
Let
\begin{equation}
\label{DivHDom}
\mathscr{V}_{H}=
\lbrace{}\zeta\in\mathbb{C}:\,\varlimsup_{\varepsilon\to0}|H_{\varepsilon}(\zeta)|=\infty\rbrace\,.
\end{equation}
The set \(\mathscr{V}_{H}\) is said to be \emph{the divergence set for the functional family
\(\{H_{\varepsilon}\}_{\varepsilon>0}\)}.
\end{definition}{\ }\\

\noindent
In view of \eqref{SySeH}, the set \(\mathscr{V}_{H}\) possesses the properties
\begin{equation}
\label{PrH}
\mathscr{V}_{H}+k\pi=\mathscr{V}_{H},\ \forall\ k\in\mathbb{Z},\ \ \ \ -\mathscr{V}_{H}=\mathscr{V}_{H}\,.
\end{equation}
This means that if \(\zeta\in\mathscr{H}\), then \((\zeta+k\pi)\in\mathscr{H}\) for any \(k\in\mathbb{Z}\)
and \((-\zeta)\in\mathscr{H}\).\\

\noindent
The equality \eqref{PuTr} for the functions \(h_{\varepsilon}(\,.\,)\)  and \(H_{\varepsilon}(\,.\,)\) implies the following statement:
\begin{lemma}
For the divergence sets \(\mathscr{V}_{f}\) and \(\mathscr{V}_{h}\), the following relation%
\footnote{\,The relation \eqref{ESRe} means that \(\zeta\in\mathscr{V}_{H}\) if and only if
\(z=e^{2i\zeta}\in\mathscr{V}_{h}\).}
 holds:
\begin{equation}
\label{ESRe}
\mathscr{V}_{h}=\exp\{2i\mathscr{V}_{H}\}
\end{equation}
\end{lemma}

Let us study the series on the right hand side of \eqref{PuTrF}.
\begin{lemma}
\label{StSer}
Let \(K\) be a compact subset of the complex plane \(\mathbb{C}\).
There exists a number \(N=N(K),\ 1\leq{}N(K)<\infty\), which depends only on \(K\) but not on \(\zeta\) such that
the inequality
\begin{equation}
\label{LiIn}
\textup{Re}\,(\zeta-n\pi)^{2}\geq{}(n\pi)^2/2
\end{equation}
holds for every \(\zeta\in{}K\) and for every \(n\in\mathbb{Z}\) such that \(|n|>N(K)\).
\end{lemma}
\begin{proof}\!Since \((\zeta-n\pi)^{2}=(n\pi)^2\big(1-\zeta/(n\pi)\big)^2\!\!,\) the equality
holds
\begin{equation*}
\textup{Re}\,(\zeta-n\pi)^2=(n\pi)^2\cdot\textup{Re}\,(1-~\zeta/(n\pi))^2\,.
\end{equation*}
If \(a\in\mathbb{C},\,|a|\leq\frac{1}{2}\), then \(\textup{Re}\,(1-a)^2\geq(1-|a|)^2\).
Hence \(\textup{Re}\,(1-a)^2\geq (1-1/4)^2>1/2\) if \(|a|\leq{}1/4\).
Thus the assertion of  Lemma
holds with \\[-2.0ex]
\begin{equation}
\label{NK}
N(K)=4\,\displaystyle{\max_{\zeta\in{}K}|}\zeta|/\pi\,.
\end{equation}
\end{proof}
\begin{corollary}\label{CorCon}The series on the right hand side of \eqref{PuTrF} converges for every \(\zeta\in\mathbb{C}\). The convergence of this series is locally uniform with respect to
\(\zeta\).
\end{corollary}

Given \(\zeta\in\mathbb{C}\), we split the series on the right hand side of \eqref{PuTrF} into two
series:
\begin{equation}
\label{Spl}
H_{\varepsilon}(\zeta)=H_{\varepsilon}^{1}(\zeta)+H_{\varepsilon}^{2}(\zeta),
\end{equation}
where
\begin{equation}
\label{Spl12}
H_{\varepsilon}^{j}(\zeta)=\sqrt{\tfrac{\pi}{\varepsilon}}\!\!\!\!
\sum\limits_{n\in{}Z_j(\zeta)}\!\!\text{\large\(e\)}^{-(\zeta-n\pi)^{2}/\varepsilon},\ \ j=1,\,2,
\ \varepsilon>0\,.
\end{equation}
and
\begin{equation}
\label{Zl12}
Z_1(\zeta)=\lbrace{}n\in\mathbb{Z}:\,\textup{Re}\,(\zeta-n\pi)^2\leq0\rbrace,\ \
Z_2(\zeta)=\lbrace{}n\in\mathbb{Z}:\,\textup{Re}\,(\zeta-n\pi)^2>0\rbrace.
\end{equation}
It is clear that
\(Z_1(\zeta)\cap{}Z_2(\zeta)=\emptyset,\  Z_1(\zeta)\cup{}Z_2(\zeta)=\mathbb{Z}.\)
From Lemma \ref{StSer} it follows that \(n\in{}Z_2(\zeta)\) if \(|n|>4|\zeta|/\pi\).
So \emph{the set \(Z_1(\zeta)\) is always finite, may be empty}.

\begin{remark}
\label{EZ}
If the set \(Z_1(\zeta)\) is empty, we set \(H_{\varepsilon}^{1}(\zeta)\equiv0\). So
the equality \eqref{Spl} holds always, whatever the set \(Z_1(\zeta)\) is, empty or not.
\end{remark}

\begin{lemma}
\label{H2}
Whatever \(\zeta\in\mathbb{C}\) is, the equality
\begin{equation}
\label{H2E}
\lim\limits_{\varepsilon\to+0}H_{\varepsilon}^{2}(\zeta)=0\,.
\end{equation}
holds.
\end{lemma}
\begin{proof}
Since \(\textup{Re}\,(\zeta-n\pi)^2>0\) for every \(n\in{}Z_2(\zeta)\), there exists a constant
\(c(\zeta)>0\) such that \(\textup{Re}\,(\zeta-n\pi)^2\geq{}c(\zeta)(n^2+1)\)
for \(n\in{}Z_2(\zeta),\,|n|\leq4|\zeta|/\pi\).
According to Lemma \ref{StSer}, the inequality
\begin{math}
\textup{Re}\,(\zeta-n\pi)^2\geq\frac{\pi^2}{8}(n^2+1)
\end{math}
holds for every \(|n|>4|\zeta|/\pi\). Therefore the inequality
\begin{equation}
\label{EZ2}
\textup{Re}\,(\zeta-n\pi)^2\geq{}c(\zeta)(n^2+1)
\end{equation}
holds for every \(n\in{}Z_2(\zeta)\) with some \(c(\zeta)>0\) which does not depend on \( n\).
Hence
\begin{equation}
\label{EstH2}
\big|H_{\varepsilon}^{2}(\zeta)\big|\leq
\sqrt{\tfrac{\pi}{\varepsilon}}\!\!\!\!
\sum\limits_{n\in{}Z_j(\zeta)}\!\!\text{\large\(e\)}^{-c(\zeta)(n^2+1)/\varepsilon}\leq
\sqrt{\tfrac{\pi}{\varepsilon}}\!\!\!\!
\sum\limits_{-\infty<n<\infty}\!\!\text{\large\(e\)}^{-c(\zeta)(n^2+1)/\varepsilon}\,.
\end{equation}
Let us estimate the value on the right hand side of \eqref{EstH2}. Clearly
\begin{equation*}
\sqrt{\tfrac{\pi}{\varepsilon}}\!\!\!\!
\sum\limits_{-\infty<n<\infty}\!\!\text{\large\(e\)}^{-c(\zeta)(n^2+1)/\varepsilon}=
\sqrt{\tfrac{\pi}{\varepsilon}}\text{\large\(e\)}^{-c(\zeta)/\varepsilon}+2
\sqrt{\tfrac{\pi}{\varepsilon}}\text{\large\(e\)}^{-c(\zeta)/\varepsilon}\!\!\!\!
\sum\limits_{1\leq{}n<\infty}\text{\large\(e\)}^{-c(\zeta)n^2/\varepsilon}
\end{equation*}
and
\begin{equation*}
\sum\limits_{1\leq{}n<\infty}\text{\large\(e\)}^{-c(\zeta)n^2/\varepsilon}\leq
\int\limits_{0}^{\infty}\text{\large\(e\)}^{-c(\zeta){}x^2/\varepsilon}dx=
\frac{1}{2}\sqrt{\frac{\pi\varepsilon}{c(\zeta)}}\,.
\end{equation*}
Thus
\begin{equation}
\label{FiEst}
\big|H_{\varepsilon}^{2}(\zeta)\big|\leq
\sqrt{\tfrac{\pi}{\varepsilon}}\text{\large\(e\)}^{-c(\zeta)/\varepsilon}+
\tfrac{\pi}{\sqrt{c(\zeta)}}\text{\large\(e\)}^{-c(\zeta)/\varepsilon}\,,
\end{equation}
and \eqref{H2E} holds.
\end{proof}
\begin{lemma}
\label{H1}
If \(Z_1(\zeta)\not=\emptyset\), then
\begin{equation}
\label{H1E}
\varlimsup_{\varepsilon\to+0}|H_{\varepsilon}^{1}(\zeta)|=\infty.
\end{equation}
\end{lemma}
\noindent
\textit{Proof}.
Given \(\zeta\in\mathbb{C}\), the numbers \(-(\zeta-n)^2\), where \(n\) runs over \(\mathbb{Z}\),
need not be pairwise different. However, for fixed \(\zeta\in\mathbb{C}\), each number can appear
among the numbers \(\{-(\zeta-n)^2\}_{n\in\mathbb{Z}}\) not more than twice:  the mapping
\(w\to(w-\zeta)^2\) covers the punctured plane \(\mathbb{C}\setminus\{\zeta\}\) twice.
 Let \(p\) be the total number of pairwise different numbers \(-(\zeta-n)^2,\,n\in{}Z_1(\zeta)\). Since \(Z_1(\zeta)\) always is finite, \(p<\infty\). Since  \(Z_1(\zeta)\not=\emptyset\), \(p>0\). Let
\(\lambda_k,\,1\leq{}k\leq{}p\), be pairwise different representatives  of the numbers \(-(\zeta-n)^2,\,n\in{}Z_1(\zeta)\), the number \(\lambda_k\) appears in the set \(-(\zeta-n)^2,\,n\in{}Z_1(\zeta)\), with
 multiplicity\footnote{So \(\sum\limits_{1\leq{}k\leq{}p}n_k=|Z_1(\zeta)|.\)} \(n_k\), where \(n_k\) is either \(1\), or \(2\).
 Denoting \(\tau=1/\varepsilon\), we present the value \(H_{\varepsilon}^{1 }(\zeta)\) as
\begin{equation}
H_{\varepsilon}^{1}(\zeta)=
\sqrt{\pi\tau}\sum\limits_{1\leq{}k\leq{}p}n_k\,{\text{\large\(e\)}}^{\lambda_k\tau},
\end{equation}
where the numbers \(\lambda_k,\,1\leq{}k\leq{}p\), are pairwise different,
\(\textup{Re}\,\lambda_k\geq0\), and \(n_k\) is either 1, or 2. Since the factor \(\sqrt{\tau}\) \ tends to \(\infty\) as \(\tau\to\infty\), it is enough to prove that
\begin{equation}
\label{ETP}
0<\varlimsup_{\tau\to\infty}\big|T(\tau)\big|\leq+\infty\,,
\end{equation}
where
\begin{equation}
\label{TP}
T(\tau)=\sum\limits_{1\leq{}k\leq{}p}n_k\,{\text{\large\(e\)}}^{\lambda_k\tau}\,.
\end{equation}

Lemma \ref{H1} is a consequence of the following fact.
\begin{lemma}
\label{EL}
Let \(T(\tau)\) be a trigonometric polynomial of the form \eqref{TP}, where
\begin{equation}
\textup{Re}\,\lambda_k\geq0,\,\,\lambda_k\,\,\textup{are pairwise different\ },\,\,\,n_k\not=0,\,\,1\leq{}k\leq{}p.
\end{equation}
Then the limiting relation \eqref{ETP} holds.
\end{lemma}
\begin{proof}
We order the numbers \(\lambda_k\) so that
\begin{equation*}
\textup{Re}\,\lambda_1=\textup{Re}\,\lambda_2=
\,\,\cdots\,\,=\textup{Re}\,\lambda_q>\textup{Re}\,\lambda_{q+1}\geq\,\,\cdots\,\,
\geq\textup{Re}\,\lambda_p\,\,\,(\,\geq0\,).
\end{equation*}
Let
 \[\lambda_k=\mu_k+i\nu_k,\ \ \
\mu_k\in\mathbb{R},\,\nu_k\in\mathbb{R},\ \ \,1\leq{}k\leq{}p.\]
Denote \(\mu=\mu_1=\mu_2=\,\,\cdots\,\,=\mu_q.\) Then
\begin{equation}
\label{TPS}
T(\tau)=e^{\mu{}\tau}\big(S(\tau)+R(\tau)\big)\,,
\end{equation}
where
\begin{equation*}
S(\tau)=\sum\limits_{1\leq{}k\leq{}q}n_ke^{i\nu_k\tau},\hspace*{2.0ex}
R(\tau)=\sum\limits_{q+1\leq{}k\leq{}p}n_ke^{(\mu_k-\mu)\tau}e^{i\nu_k\tau},
\end{equation*}
and
\begin{equation}
\mu\geq0,\hspace*{2.0ex} \mu>\mu_k
\,\,\textup{ for }\,\,q+1\leq{}k\leq{}p,\hspace*{2.0ex} \nu_k\in\mathbb{R}\,\,\textup{ for }\,\,1\leq{}k\leq{}p,
\end{equation}
Since \(\mu_k<\mu\) for \(q+1\leq{}k\leq{}p\),
\begin{equation}
\label{ZL}
\lim_{\tau\to+\infty}R(\tau)=0\,.
\end{equation}
Since \(n_k\not=0\) and the numbers \(\nu_k\in\mathbb{R}\) are pairwise different, the function
\(S(\tau)\) is an almost periodic (possibly periodic) function, \(S(\tau)\not\equiv0\).
Therefore
\begin{equation}
\label{APP}
\varlimsup_{\tau\to\pm\infty}|S(\tau)|=\sup_{\tau\in\mathbb{R}}|S(\tau)|>0.
\end{equation}
The limiting relation \eqref{ETP} is a consequence of \eqref{TPS}, \eqref{ZL}, \eqref{APP}
and \(\mu\geq0\).
\end{proof}

From definition of the divergence set \(\mathscr{V}_H\) (see Definition \ref{DivHDom}),
from Lemmas \ref{H1}, \ref{H2}, and from the equality \eqref{Spl} we obtain
the following description of the divergence set \(\mathscr{V}_H\).
\begin{theorem}
\label{DDS}
\begin{equation}
\label{DDSE}
\mathscr{V}_H=\lbrace\zeta\in\mathbb{C}:\,Z_1(\zeta)\not=\emptyset\rbrace\,,
\end{equation}
where \(Z_1(\zeta)\) is defined in \eqref{Zl12}.
\end{theorem}

Now we obtain a geometric description of the divergence set \(\mathscr{V}_H\). From this description
Theorem \ref{DivTh} follows.

For \(n\in\mathbb{Z}\), let
\begin{equation*}
Q_n=\lbrace\zeta\in\mathbb{C}:\,\textup{Re}\,(\zeta-n\pi)^{2}>0\rbrace\,.
\end{equation*}
According to \eqref{Zl12} and Theorem \ref{DDS},
\begin{equation}
\label{GD1}
\mathbb{C}\setminus\mathscr{V}_H=\bigcap_{n\in\mathbb{Z}}Q_n.
\end{equation}
It is clear that
\begin{equation*}
Q_n=Q_0+n\pi, \,\,\forall\,n\in\mathbb{Z}\,,
\end{equation*}
where
\begin{equation*}
Q_0=\lbrace\zeta\in\mathbb{C}:\,-\pi/4<\textup{arg}\zeta<\pi/4\rbrace
\bigcup\lbrace\zeta\in\mathbb{C}:\,
3\pi/4<\textup{arg}\zeta<5\pi/4\rbrace\,.
\end{equation*}
Simple geometric considerations show that
\begin{equation}
\label{UI}
\bigcap_{n\in\mathbb{Z}}Q_n=\bigcup_{n\in\mathbb{Z}}T_n,
\end{equation}
where
\begin{equation}
\label{CoS}
T_n=T_0+n\pi,\ \ n\in\mathbb{Z},
\end{equation}
and \(T_0\) is the open square with the vertices at the points
\(\zeta=0\), \(\zeta=\pi\),\,\(\zeta=(1+i)\pi/2\),\,\(\zeta=(1-i)\pi/2\). According to
\eqref{GD1} and \eqref{UI}, the divergence set \(\mathscr{V}_H\) is
\begin{gather}
\mathscr{V}_H=\Gamma_{+}\bigcup\Gamma_{-},\\
\intertext{where}
\Gamma_{+}=\lbrace\xi+i\eta:\,-\infty<\xi<\infty,\ \ \eta\geq\phantom{-}\min_{n\in\mathbb{Z}}|\xi-n\pi|\rbrace\,,\\
\Gamma_{-}=\lbrace\xi+i\eta:\,-\infty<\xi<\infty,\ \ \eta\leq-\min_{n\in\mathbb{Z}}|\xi-n\pi|\rbrace\,.
\end{gather}
According to \eqref{ESRe},
\begin{equation}
\label{DeD}
\mathscr{V}_h=\exp\{2i\Gamma_{+}\}\bigcup\exp\{2i\Gamma_{-}\}\,.
\end{equation}
It is clear that
\begin{equation}
\label{DiD}
\exp\{2i\Gamma_{+}\}\subset\big(\{1\}\cup\mathbb{D}^{+}\big),\ \ \
\exp\{2i\Gamma_{-}\}\subset\big(\{1\}\cup\mathbb{D}^{-}\big).
\end{equation}
From \eqref{SeDe}, \eqref{DeD} and \eqref{DiD} we conclude that the divergence set \(\mathscr{V}_f\) is
\begin{equation}
\label{FSTE}
\mathscr{V}_f=\exp\{2i\Gamma_{-}\}\,.
\end{equation}
It is clear that
\[\exp\{2i\Gamma_{-}\}=\mathbb{C}\setminus\mathscr{G},\]
where the set \(\mathscr{G}\) was defined in Definition \ref{DeDo}. Thus
\begin{equation}
\label{LEq}
\mathscr{V}_f=\mathbb{C}\setminus\mathscr{G}\,.
\end{equation}
Theorem \ref{DivTh} is proved. \hfill \(\Box\)


\end{document}